\documentclass[11pt]{article}
\usepackage{enumerate}
\usepackage{fullpage}
\usepackage{amsmath,amsthm,verbatim,amssymb,amsfonts,amscd, graphicx,color}
\usepackage{mathptmx}
\usepackage{subcaption}

\newtheorem{theo}{Theorem}[section]
\newtheorem{example}[theo]{Example}
\newtheorem{lem}[theo]{Lemma}
\newtheorem{prop}[theo]{Proposition}
\newtheorem{definition}[theo]{Definition}
\newtheorem{rem}[theo]{Remark}

\newcommand{\bd}{\begin{displaymath}}
\newcommand{\ed}{\end{displaymath}}
\newcommand{\be}{\begin{equation}}
\newcommand{\ee}{\end{equation}}
\newcommand{\bea}{\begin{eqnarray}}
\newcommand{\eea}{\end{eqnarray}}
\newcommand{\bda}{\begin{eqnarray*}}
\newcommand{\eda}{\end{eqnarray*}}
\newcommand{\ba}{\begin{array}}
\newcommand{\ea}{\end{array}}

\newcommand{\B}{I\kern -.35em B}
\newcommand{\gr}{\mbox{\rm graph\,}}
\newcommand{\R}{\mathbb R}
\renewenvironment{proof}[1]{%
\par\vspace{1\baselineskip}%
\noindent{\em Proof#1.\ \ }\ignorespaces }{%
\nobreak\hfill\mbox{\ \ $\Box$}%
\par\vspace{1\baselineskip}%
}

\newcommand{\Mt}{\Rightarrow}

\newcommand{\st}{\subset}

\newcommand{\sm}{\setminus}

\newcommand{\U}{\mathcal U}

\newcommand{\di}{\mbox{\rm dist}}
\newcommand{\E}{{\cal E}}

\newcommand{\e}{\varepsilon}
\newcommand{\al}{\alpha}
\newcommand{\ph}{\varphi}
\newcommand{\sth}{ \, :\;}

\newcommand{\dd}{\mbox{\rm\,d}}

\renewcommand{\lll}{\langle}
\newcommand{\rrr}{\rangle}

\newcommand{\p}{{\mathcal P}}

\newcommand{\Y}{\mathcal Y}
\newcommand{\Z}{\mathcal Z}

\def\R{\mathbb{R}}
\newcommand{\To}{\rightrightarrows }

\newcommand{\tu}{{\tilde u}}
\newcommand{\hu}{{\hat u}}
\newcommand{\bino}{\bigskip\noindent}

\title{On the accuracy of the model predictive control method\thanks{This research is supported by 
the Austrian Science Foundation (FWF) under grant No I4571.}}

\author{
Georgi Angelov\thanks{Institute of Statistics and Mathematical Methods in Economics,
Vienna University of Technology, Austria, {\tt georgi.angelov@tuwien.ac.at}} 
\and  
Alberto Dom\'inguez Corella\thanks{Institute of Statistics and Mathematical Methods in Economics,
Vienna University of Technology, Austria, {\tt alberto.corella@tuwien.ac.at}} 
\and  
Vladimir M. Veliov\thanks{Institute of Statistics and Mathematical Methods in Economics,
Vienna University of Technology, Austria, {\tt vladimir.veliov@tuwien.ac.at}} }

\date{}
\begin{document}
	\maketitle

\begin{abstract}
The paper investigates the accuracy of the Model Predictive Control (MPC) method for finding online approximate optimal feedback control for Bolza type problems on a fixed finite horizon. The predictions for the dynamics, the state measurements, and the solution of the auxiliary open-loop control
problems that appear at every step of the MPC method may be inaccurate. The main result provides an error estimate of the MPC-generated solution compared with the optimal open-loop solution of the “ideal”
problem, where all predictions and measurements are exact. The technique of proving the estimate involves an extension of the notion of strong metric sub-regularity of set-valued maps and utilization of a
specific new metric in the control space, which makes the proof non-standard. The result is specialized
for two problem classes: coercive problems, and affine problems.
\end{abstract}
\medskip



\section{Introduction} \label{SIntro}

Model Predictive Control (MPC) is a powerful method for approximate on-line feedback control, widely used in 
industrial applications and recently  in digital engine control and microelectronics, see, e.g., 
\cite{G+P-11,MPOC_HB-18,Wolf-16}. On the other hand, the rigorous mathematical theory investigating 
the scope of validity and the efficiency of the MPC method under appropriate 
assumptions is still underdeveloped. The present paper contributes to this theory by investigating the accuracy of a version
of the MPC method applied to a finite horizon optimal control problem (the so-called {\em economic MPC} with shrinking 
horizon). 

To set the stage, we consider the following optimal control problem, further denoted by $\p_p(0,x_0)$:
\be \label{EOF}
  \min_{u\in \U} \Big\{ J_p(u) := g_T(x(T)) + \int_0^T  g(p(t),x(t),u(t)) \dd t \Big\},
\ee
subject to
\be \label{Ex}
      \dot x(t) = f(p(t), x(t), u(t))  \quad  x(0)=x_0.
\ee
Here the state vector $x(t)$ belongs to $\R^n$ and the control function $u(\cdot)$ belongs to the set $\U$ 
of all Lebesgue measurable functions $u:[0,T] \to U$, where $U\st \R^m$. The function $p$ represents 
an uncertain time-dependent parameter which is known to belong to 
a set $\varPi$ of bounded Lebesgue measurable functions $p:[0,T] \to \R^l$.
Correspondingly, $f:\R^l\times\R^n\times\R^m\to\R^n$, $g:\R^l\times\R^n\times\R^m\to\R$, $g_T:\mathbb R^n\to\mathbb R$. 
The initial state $x_0 \in \R^n$ and the final time $T>0$ are fixed.

A version of the MPC method (called further MPC algorithm) applied to the above problem  
is described in detail in Subsection \ref{SS_MPC}. Here we briefly present the main result given in Theorem \ref{Tmain},
Subsection \ref{SS_main}. 
It is assumed that for some particular parameter function, $\hat p$, equation (\ref{Ex}) represents 
a real system, therefore problem (\ref{EOF})--(\ref{Ex}) with $p = \hat p$ (that is, problem $\mathcal P_{\hat p}(0,x_0)$)
is called {\em reference problem}. However, $\hat p$ and the initial state $x_0$ are not assumed to be exactly known.
The MPC algorithm generates a control function by solving a sequence of auxiliary open-loop optimal control problems.
Given a mesh $0 = t_0 < t_1 < \ldots < t_N = T$, the auxiliary problem at the $k$-th step is of the same kind as 
(\ref{EOF})--(\ref{Ex}), but on the shorter time-interval $[t_k, T]$.
A prediction $p$ for the parameter function on $[t_k,T]$ is given, 
and the initial state at $t_k$ is obtained by measuring the real systems state at time $t_k$. 
Both the prediction and the measurement 
may be inaccurate. In addition, the auxiliary problem at the $k$-th step
is only approximately solved, which is another source of error. The approximate optimal control in the $k$-th auxiliary problem
is only applied to the `'real'' system on the interval $[t_k, t_{k+1}]$, then the procedure is further repeated
on the next interval, resulting at the end  in what is called the MPC- generated control.
\begin{rem}
There is an amount of literature related to robustness of the MPC method, see e.g.\ \cite{Langson+Chryssochoos+Rakovic+Mayne-2004,Grimm+Messina+Tuna+Teel-2007}
and several chapters in \cite{MPOC_HB-18}. Stabilization or target problems for discrete dynamics are usually 
considered for particular classes of problems, and various notions of robustness and approaches 
are investigated – see e.g.\ the overview paper \cite[Chapter by Rakovi\'c]{Baillieul+Samad-2020}. The result in the 
present paper, as explained above, is more related to the paper \cite{Grune+Palma-2014}. In that paper the accuracy 
of the MPC method is investigated (for a general class of discrete-time problems) in terms of the 
objective value only. Of course, this is the primal accuracy criterion; however, it is of interest to know 
how close the MPC-generated control and trajectory are to the optimal open-loop one for the case 
of exact and complete information. The present paper contributes to this issue and, most 
importantly, employs and further develops the approach based on the general property of strong 
sub-regularity of a map associated with the problem under consideration.
\end{rem}

The main result in the paper gives an estimate of the difference
between the MPC-generated control and the optimal open-loop control for the reference problem
(the latter corresponding to the ``ideal'' scenario where the prediction, the measurement, and the solution 
of the auxiliary problems are all exact). A remarkable feature of the estimation is that the overall error of 
the MPC-generate control depends on the {\em average} of the errors appearing at the steps of the algorithm, 
thus occasional relatively large errors in the prediction or measurement do not substantially damage the MPC-generated control.
Another interesting feature of the overall error is that for some classes of problems it depends linearly on the averaged 
errors appearing at the steps of the method, while for other classes, the estimate of the overall error 
depends on the square root of the averaged errors (and this estimate is sharp).
We mention that an estimate of the difference between the MPC-generated control 
and the realization of the optimal {\em feedback} control in the reference problem 
is obtained \cite{AD+IK+MK+VV+PV-20}. In \cite{AD+IK+MK+VV-20}, this result is extended to 
a comparison with  the optimal open-loop control in the reference problem, as in the present paper.
However, in both quoted papers the results are obtained within a much more restrictive framework:
a single prediction is used which does not change from step to step, the results only apply to the Euler discretization
of the auxiliary problems, and most importantly, the reference optimal control problem is assumed {\em coercive}
(see Subsection  \ref{SS_coerc} for the notion). In fact, the main goal of the present paper is to extend 
the results about the accuracy of the MPC method beyond the coercive case, especially for affine control
problems. 

The main result (the error estimate in Theorem \ref{Tmain}) is obtained under general assumptions;
the most demanding one is the requirement that the 
map associated with the first order optimality conditions (called {\em optimality map})
for the reference problem is {\em strongly metrically sub-regular}
in an appropriate space setting. In Subsection \ref{SS_SsR}, we extend the abstract notion of  
strong metric sub-regularity of a set-valued map by involving two metrics in the domain of  the map.
In Subsection \ref{SSOptMap}, we define the optimality map and the space setting. In the control space
(which is a projection of the domain of  the optimality map) 
we introduce a specific new metric, which is of key importance for the analysis of the MPC method
and may be useful in other contexts.  

The proof of  Theorem \ref{Tmain} (given in Appendix) is non-trivial and substantially differs from all proofs 
of error estimates in optimal control that the authors know.

Section \ref{S_conditions} presents or recalls sufficient conditions for the extended strong metric sub-regularity 
of the optimality map for two non-intersecting classes of problems: for {\em coercive problems} and for {\em affine problems}.
The paper concludes with an example where the MPC algorithm is applied to a spacecraft stabilization problem.
The numerical results confirm the theoretical error estimate and its sharpness.  

\section{Preliminaries} \label{SPrelim}

\subsection{Strong sub-regularity} \label{SS_SsR}

Here we introduce a notion which extends the property of  
{\em Strong (metric) sub-Regularity} (Ss-R)  (see, e.g., \cite[Chapter 3.9 ]{AD+TR-book2}
and the recent paper \cite{Cibulka+Dontchev+Kruger-18}). 
Namely, we consider a general metric space $\Y$ in which two metrics are defined: $d$ and $d^*$, and another metric space
$\Z$ with a metric $d_\Z$. We denote by $\B(y;\al)$ the closed ball with radius $\al \geq 0$ in $(\Y,d)$ centered at $y$,
by $\B_*(y;\al)$ the ball with radius $\al \geq 0$ in $(\Y,d^*)$, and similarly, $\B_\Z(z;\al)$ is the respective ball in $\Z$.

\begin{definition} \label{DSMsR}
A set-valued map $\Phi : \Y \To \Z$ 
is called Ss-R at $(\hat y, \hat z) \in \Y \times \Z$ (with respect to the metrics $d$ and $d^*$) 
if $\hat z \in \Phi(\hat y)$ and there are positive constants $\alpha,\beta$  and $\kappa$ 
(called further parameters of Ss-R) such that 
\bd
    d^*(y,\hat y) \, \leq \, \kappa d_\Z(z, \hat z) \quad \mbox{for all } y \in \B(\hat y;\al), \;\; z \in \Phi(y)\cap \B_\Z(\hat z;\beta).
\ed
\end{definition}

The Ss-R property plays a fundamental role in the error analysis of numerical methods. 
It was introduced under this name in \cite{AD+TR-04}, but has also been used under several other names
(see also \cite[Chapter~1]{Klatte+Kummer-2002-book} for the related but stronger property of strong upper regularity). 
A more detailed historical account can be found in \cite[Section 1]{Cibulka+Dontchev+Kruger-18}. 
The extension with two metrics in $\Y$, presented above, is essential for the applications in the present paper.

The following simple claim is a modification of \cite[Theorem 2.1]{Cibulka+Dontchev+Kruger-18} 
for the case of two metrics in $Y$.

\begin{prop} \label{P_LG}
Assume that $\Z$ is a linear space and $d_\Z$ is a shift-invariant metric in $\Z$. Assume, in addition, 
there exists a number $\gamma > 0$ such that $d(y_1,y_2) \leq \gamma d^*(y_1,y_2)$ for every $y_1, y_2 \in \Y$. 
Let $\Phi  : \Y \To \Z$ be Ss-R at $(\hat y, \hat z)$ with parameters $\alpha',\beta'$ and $\kappa'$. 
Let the positive numbers $\epsilon,\mu, \kappa, \al, \beta$ satisfy the relations
\be \label{Econst_sub}
   \al \leq \al', \quad \beta+\mu\alpha+\epsilon\le\beta',  \quad   \mu \gamma \kappa' < 1, 
           \quad  \kappa = \frac{\kappa'}{1-\mu \gamma \kappa'}.
\ee
Then for every function $\ph : \Y \to \Z$ that satisfies the conditions
\bd
   d_{\Z}(\ph(\hat y),0)\le\epsilon, \quad d_\Z(\ph(y),\ph(\hat y)) \leq \mu \,d(y,\hat y)  \quad \forall y \in \B(\hat y; \al),
\ed
the map $\ph + \Phi$ is strongly sub-regular at $(\hat y, \hat z + \ph(\hat y))$ with parameters $\alpha,\beta$ and $\kappa$. 
\end{prop} 

\begin{proof}{}
Obviously, $(\hat y, \hat z + \ph(\hat y)) \in \gr(\ph + \Phi)$. 
Let us fix arbitrarily $z \in \B_\Z(\hat z; \beta)$ and $y \in \B(\hat y;\al) \st \B(\hat y;\al')$ 
such that $z \in \ph(y) + \Phi(y)$.
Then $z - \ph(y) \in \Phi(y)$, and 
\begin{align*}
	d_\Z(z - \ph(y),\hat z)&\le d_\Z(z,\hat z)+d_\Z(\ph(y),\ph(\hat y))+d_\Z(\ph(\hat y),0)\\
	&\le\beta+ \mu\alpha+\epsilon\le \beta'.
\end{align*}
Due to the Ss-R property of $\Phi$ we estimate
\bda
      d^*(y,\hat y) & \leq &\kappa'  d_\Z(z - \ph(y), \hat z) 
          \leq \kappa' d_\Z(z,\hat z + \ph(\hat y))  + \kappa' d_\Z(\ph(y), \ph(\hat y)) \\
    &\leq& \kappa' d_\Z(z,\hat z + \ph(\hat y)) + \kappa' \mu \, d(y, \hat y) \leq 
            \kappa' d_\Z(z,\hat z + \ph(\hat y)) + \kappa' \mu \gamma\, d^*(y, \hat y) ,
\eda
which implies the claim of the theorem due to the definition of $\kappa$ in (\ref{Econst_sub}).
\end{proof}

\subsection{The optimality map} \label{SSOptMap}

Problem $\p_p(0,x_0)$ given by (\ref{EOF})--(\ref{Ex}) will be considered under the following assumptions.

\bino
{\bf Assumption (A1)}. The set $U$ is convex and compact. The functions $f,g$ and $g_T$ are two times differentiable 
with respect to $(x,u)$, these functions and their first and second derivatives in $(x,u)$ are Lipschitz continuous 
(with respect to $(p,x,u)$) with a Lipschitz constant $L$. 

\bino
For any $p \in \Pi$, along with problem $\p_p(0,x_0)$ we consider the family, denoted by $\p_p(\tau,x_\tau)$, 
consisting of problems which have the same form as (\ref{EOF})--(\ref{Ex}) but  with the initial time $0$ 
replaced with any $\tau \in [0,T)$ and $x_0$ replaced with any $x_\tau \in \R^n$. 
Of course, then only the restriction of the parameter $p$ to $[\tau,T]$ matters. 

\bino
{\bf Assumption (A2)}. 
For every $u \in \U$, $x_0 \in \R^n$, and $p \in \varPi$ equation (\ref{Ex}) has a solution $x$ on $[0,T]$ 
(which is then unique due to Assumption (A1)).
For every $\tau \in [0,T)$, $x_\tau \in \R^n$ and $p \in \varPi$  problem $\p_p(\tau,x_\tau)$ has an optimal solution. 

\bino
Since the analysis in this paper is local, only local Lipschitz continuity of the functions mentioned in Assumption (A1) 
is needed; we assume global Lipschitz continuity to avoid routine technicalities.
Assumption (A2) is also stronger than necessary, again for the sake of transparency. Local existence around a reference
parameter and trajectory (to be introduced later) suffices.

\begin{rem} \label{Rsol}{\em
Optimality in the last assumption means local optimality of the objective functional with respect to the $L^1$-norm
of the controls.  
In fact, it is only needed that any solution $(x,u)$ of $\p_p(\tau,x_\tau)$ satisfies, together with
an absolutely continuous ({\em adjoint}) function $\lambda : [\tau, T] \to \R^n$, the {\em optimality (Pontryagin) system}
}\end{rem}
\vspace{-0.5cm}
 \bea
        0 &=&   - \dot x(t) + f(p(t),x(t), u(t)), \quad x(\tau) - x_\tau = 0,   \label{EPx}\\
      0 &=&  \dot \lambda (t) + \nabla_{\!\! x} H(p(t),x(t),\lambda(t),u(t)), \quad \lambda(T) = \nabla g_T(x(T)),  \label{EPlambda} \\
        0 &\in& \nabla_{\!\! u} H(p(t),x(t),\lambda(t),u(t)) + N_U(u(t)),   \label{EPu}  
\eea
where 
\bd
     N_U(u) := \left\{ \begin{array}{ll} \{q \in \R^n  \mid {\langle q, v-u \rangle} \leq 0
      {\text{ for all } v \in U}\} & {\text{if } u \in U,}\\
             \emptyset & \text{if } u \notin U  \end{array}\right.
\ed
is the normal cone to $U$ at $u$, and the Hamiltonian $H$ is defined as usual:
\bd
                 H(p,x,\lambda,u) := g(p,x,u) + \lll \lambda, f(p,x,u) \rrr.
\ed

Next, we reformulate the optimality system in functional spaces. 
The space $L^q(\tau,T)$, $q =1, 2, \ldots, \infty$, of vector functions on $[\tau,T]$ (with any fixed dimension)
has the usual meaning, with the norm denoted by $\| \cdot \|_q$.  The space of all absolutely continuous vector functions
on $[\tau,T]$  is denoted by $W^{1,1}(\tau,T)$, with the norm $\| x \|_{1,1} = \| x \|_1 + \| \dot x \|_1$.
The notations of norms do not include the time horizon, but it will be clear from the context.
For the same reason we often skip the time horizon from the notations of spaces.

Denote
\bd
       Y_\tau := W^{1,1}(\tau, T) \times W^{1,1}(\tau,T) \times \U_\tau, 
          \qquad Z_\tau := L^1(\tau,T) \times \R^n \times L^1(\tau,T)\times\mathbb R^n\times L^\infty(\tau,T),    
\ed
where $\U_\tau = \{ u \in L^\infty(\tau,T) \sth u(t) \in U \mbox{ for a.e. } t \in [\tau,T]\}$ is the set of admissible control functions 
on $[\tau,T]$, thus $\U_0 = \U$. We also set $Y := Y_0$ and $Z := Z_0$.
The metrics in $Y_\tau$ and $Z_\tau$
are given in terms of norms as follows: for $y = (x,\lambda,u) \in Y_\tau$ and $z = (\xi, \nu, \eta,\pi, \rho) \in Z_\tau$
\bd
       d(y,0) := \| y \| := \| x \|_{1,1} + \| \lambda \|_{1,1} + \| u \|_1, \qquad
        d_Z(z,0) := \| z \|  := \| \xi \|_{1} + | \nu | + \| \eta \|_{1} +  |\pi|+ \| \rho \|_\infty.
\ed
In addition, we define in $Y$ a second metric, $d^*$, as follows. 
Let $\Gamma \st [0,T]$ be a fixed finite set.
For $u_1, u_2 \in \U_\tau$, denote 
\be\label{EGam0}
    d^*(u_1,u_2) := \inf \{ \e > 0 \sth |u_1(t) - u_2(t) | \leq \e \; 
    \mbox{ for a.e. } \; t \in [0,T] \sm (\Gamma + [-\e,\e]) \}.
\ee 
Somewhat overloading the notation, we define in $Y$ the shift-invariant metric
\bd
          d^*(y,0) := \| x \|_{1,1} + \| \lambda \|_{1,1} + d^*(u,0).
\ed

\begin{lem} \label{L_d<d}
For every $u_1, u_2 \in \U$ it holds that
\bd
    \| u_1 - u_2 \|_1 \leq \gamma \; d^*(u_1,u_2), 
\ed
where $\gamma := \max\{ 1, T+2M \, {\rm diam} \,(U)\}$ and $M$ is the number of points in $\Gamma$ . 
\end{lem}

The proof is straightforward. 
Since $\gamma \geq 1$, we also have $\| y \| \leq \gamma \, d^*_\tau(y,0)$ for any $y \in Y_\tau$.

Below we use the same notation for the Nemytskii operator and for its generating function:
$f(p,x,u)(t) = f(p(t),x(t),u(t)$, $\nabla_{\!\! x} H(p,x,\lambda,u)(t) = \nabla_{\!\! x} H(p(t),x(t),\lambda(t),u(t))$, etc.
For any $p \in \varPi$, $\tau \in [0,T)$, and $x_\tau \in \R^n$ define on $Y_\tau$ the set-valued map
\bd 
   \Phi_{(p,\tau,x_\tau)}(y) = F_{(p,\tau,x_\tau)}(y) + \left( \begin{array}{c}
                                                      0 \\
                                                       0 \\
                                                      0 \\
                                                      0\\
                                                   N_{\U_\tau}(u) 
                                 \end{array} \right),
                \qquad  F_{(p,\tau,x_\tau)}(y) = \left( \begin{array}{c}
                                            - \dot x + f(p,x,u)   \\
                                            x(\tau) - x_\tau \\
                                             \dot \lambda + \nabla_{\!\! x} H(p,x,\lambda,u) \\
                                             \lambda(T)-\nabla g_T(x(T))\\
                                        \nabla_{\!\!u} H(p,x,\lambda,u) 
                                         \end{array} \right),
\ed
where now $N_{\U_\tau}(u)$ is the normal cone to $\U_\tau$ at $u$ in the space $L^1(\tau, T)$, that is,
\bda
     N_{\U_\tau}(u) &:=& \left\{ \begin{array}{ll} \{l \in L^\infty(\tau,T)  \sth \int_\tau^T \lll l(t), v(t) -u(t) \rrr \dd t 
     \leq 0 {\text{ for all } v \in \U_\tau}\} & {\text{if } u \in \U_\tau,}\\
             \emptyset & \text{if } u \notin \U_\tau\end{array}\right. \\
    && = \{ l \in L^\infty(\tau,T) \sth l(t) \in N_U(u(t)) \mbox{ for a.e. } t \in [\tau, T] \}.
\eda
With these notations one can recast the optimality system for problem  $\p_p(\tau,x_\tau)$ as
\bd
          0 \in \Phi_{(p,\tau,x_\tau)}(x,\lambda,u),
\ed 
therefore the map $\Phi_{(p,\tau, x_\tau)}$ is called {\em optimality map}. Obviously, due to the compactness 
of the set $U$,  $\Phi_{(p,\tau,x_\tau)}$ is a set-valued map from  $Y_\tau$ to $Z_\tau$.

\bino
Let us fix a reference parameter $\hat p \in \varPi$ and denote by $(\hat x, \hat u)$ a solution of problem
$\p_{\hat p}(0,x_0)$ (see Remark~\ref{Rsol}). Let $\hat \lambda$ be the corresponding adjoint function, 
so that the triplet $\hat y := (\hat x, \hat \lambda, \hat u)$ satisfies the optimality system (\ref{EPx})--(\ref{EPu})
with $\tau = 0$ and $x_\tau = x_0$,  or equivalently, the inclusion $0 \in \Phi_{(\hat p,0,x_0)}(\hat y)$.
The following assumption plays a key role in the error analysis of the MPC method presented in the next section. 

\bino
{\bf Assumption (A3).} The map $\Phi_{(\hat p,0,x_0)} : Y \To Z$ is strongly sub-regular 
at $(\hat y,0)$ (in the metrics $\| \cdot \|$ and $d^*$ in $Y$) with parameters $\hat \al,\hat \beta$ and $\hat \kappa$.

\bino
The finite set $\Gamma$ that appears in the definition of the metric $d^*$ is arbitrary, but will be appropriately 
specified in Section \ref{S_conditions} for several classes of problems, together with sufficient conditions for (A3).

\section{The accuracy of the model predictive control method} \label{S_MPC}

This section presents the main result: the estimate of the accuracy of the model predictive control method,
beginning with the description of the method in the context of finite horizon optimal control.

\subsection{The model predictive control method} \label{SS_MPC}

The MPC, applied to optimal control problems containing uncertain parameters, is a method 
for approximation of an optimal feedback control in real time by successively solving open-loop 
optimal control problems. Each of these open-loop problems involve measurements of the current system
state and predictions for the uncertain parameters. In the next three paragraphs we present a version of the MPC method.

The optimal control problem into question is problem $\p_p(0,x_0)$, considered under Assumptions (A1)--(A3),
where the parameter function $p \in \varPi$ is uncertain.
It is assumed that for some parameter $\hat p \in \varPi$ equation (\ref{Ex}) with $p = \hat p$ 
reproduces a ``real" system, the states of which can be measured (with a measurement
errors). As in the previous subsections we denote by $(\hat x, \hat u)$ a reference optimal solution 
of $\p_{\hat p}(0,x_0)$. 

Given a natural number $N$, we denote by $\{t_k\}_{k=0}^N$ the grid with step-size $h=T/N$, that is, $t_k=kh$, $k=0,\ldots,N$. The MPC algorithm generates in real time an admissible control function, denoted further by $u^N$. It is applied to the ``real'' system, that is, \eqref{Ex} with $p=\hat p$, resulting in a ``real'' trajectory $x^N$.

At time $t_k$, $k=0,\ldots,N-1$, the algorithm does the following.
\begin{enumerate}
\item Measure the state $x^N(t_k)$ with error $e_k$, that is, the vector $x_k^0=x^N(t_k)+e_k$ becomes available.
\item Make a prediction $p_k\in\Pi$ for the time horizon $[t_k,T]$.
\item Find an approximate solution $(\tilde x_k,\tilde u_k)\in W^{1,1}\times \U_{t_k}$ of the problem $\mathcal P_{p_k}(t_k,x_k^0)$.
\item Define the control $u^N$ as $u^N(t)=\tilde u_k(t)$ on $(t_k,t_{k+1}]$ and apply to the ``real'' system on this interval.
\end{enumerate}

The process continues in the same way as long as $k < N$. The control $u^N$ is called {\em MPC-generated control}
and the corresponding trajectory $x^N$ of the ``real" system (\ref{Ex}) with $u = u^N$ and $p = \hat p$
is called {\em MPC-generated trajectory}.

\bino
Two points are to be clarified. First, the quality of a prediction $p_k \in \varPi$ on $[t_k,T]$ will be measured by the 
norm $e_k^p :=  \| p_k - \hat p_{[t_k,T]}\|_\infty$. Second, the pair $(\tilde x_k, \tilde u_k)$ is an approximate solution  of
problem $\p_{p_k}(t_k, x_k^0)$ in the sense that for some absolutely continuous $\tilde \lambda_k$
the triplet $\tilde y_k := (\tilde x_k, \tilde \lambda_k, \tilde u_k)$
satisfies the inclusion (approximate optimality conditions) 
\be \label{LRrez}
          \tilde z_k \in \Phi_{(p_k,t_k, x_k^0)}(\tilde y_k)
\ee
with some $\tilde z_k \in Z_\tau$. 
We mention that most of the numerical methods for optimal control give approximations 
with a small residual $\tilde z_k$. The norm $e_k^{u} := \| \tilde z_k \|$ of the residual will be used 
as a measure of the accuracy of the approximate solution  $(\tilde x_k,\tilde u_k)$ of problem $\p_{p_k}(t_k, x_k^0)$.  
 
\subsection{The main theorem} \label{SS_main}

The formulation of the main theorem uses the notations $e_k$, $e^p_k$, $e^{u}_k$, $\tilde u_k$, $u^N$ and 
$x^N$ introduced in the description of the MPC algorithm. In particular, $(x^N, u^N)$ is the MPC-generated
trajectory-control pair, which is compared in the next theorem with the reference optimal open-loop solution $(\hat x, \hat u)$
of the ``real'' problem $\p_{\hat p}(0,x_0)$.

\begin{theo} \label{Tmain}
Let Assumptions (A1)--(A3) be fulfilled. 
Then there exists numbers $N_0$, $\delta > 0$, $C_1$, $C_2$, and $C_3$ such that 
for any natural number $N \geq N_0$, for any sequence of measurement errors $\{e_k\}$, 
for any sequence of predictions $p_k \in \varPi$ and approximation errors $\{e_k^{u}\}$ satisfying the conditions
\bd
        | e_k | + e_k^p + e_k^{u} \leq \delta, \qquad 
           \| \tilde u_k - \hat u\|_1 \leq \delta, \quad k= 0, \ldots, N-1,
\ed 
any MPC-generated trajectory-control pair $(x^N, u^N)$ satisfies the estimate  
\bd
        \| u^N - \hat u \|_1  +  \| x^N - \hat x \|_{1,1}  \, \leq \, \left\{ \begin{array}{cl} 
                         C_1{\cal E} & \mbox{ if  } \; \Gamma = \emptyset, \\
                         C_2  \sqrt{ \cal E} + C_3 h & \mbox{ if  } \; \Gamma \not= \not\emptyset,
             \end{array} \right.
\ed
where 
\bd
        {\cal E} := \frac{1}{N} \sum_{k=0}^{N-1} (|e_k| + e^{p}_k + e^{u}_k) 
\ed 
is the averaged error appearing at the MPC steps.
\end{theo}
 
\bino
The proof of the theorem is postponed to Section \ref{Sproof}.
Below in this subsection we discuss the obtained result and the assumptions.

\begin{rem} \label{R_A3} About Assumption (A3). {\em
Assumptions (A1) and (A2) are standard and non-restrictive, although somewhat stronger than necessary, as 
noted after their formulation. Assumption (A3) has to be explained. 

First of all, what is the finite set $\Gamma$ in the definition of the metric $d^*$ which is involved in (A3) through
the definition of strong sub-regularity? 
This set may depend on the reference optimal control $\hat u$ of the unperturbed problem $\p_{\hat p}(0, x_0)$.
Presumably, it consists of points of discontinuity of $\hat u$, but may be larger in order to include points at which 
the optimal control of a slightly disturbed problem may be discontinuous. Example \ref{Example} below illustrates 
this situation. The meaning of the metric $d^*$ is that the distance between two control functions is small in this
metric when their values are close to each other, possibly excepting points that are close to the set $\Gamma$. 
This property of the metric with which the Ss-R assumption (A3) is fulfilled is of key importance for that 
convergence and error analysis of the MPC method.

In the next section we shall provide sufficient conditions under which Assumption (A3) is fulfilled in particular classes of
problems with empty or non-empty set $\Gamma$. 
}\end{rem}

\begin{rem} \label{R_theo} Discussion on the theorem. {\em
Theorem \ref{Tmain} estimates the error of the MPC-generated solution, compared
with the optimal solution of the reference (unperturbed) problem, caused 
by prediction errors $\{e^p_k\}_k$, measurement errors $\{e_k\}_k$,
approximation errors $\{ e^{u}_k \}_k$, and the sampling size $h$.
An important point is that the error estimate in the theorem depends on the average error, $\E$,
which means that relatively large errors may occasionally appear at some MPC steps
without a substantial influence on $\E$. 

A similar result as in Theorem \ref{Tmain} is obtained in \cite{AD+IK+MK+VV-20} in the case $\Gamma = \emptyset$
(see Subsection \ref{SS_coerc} of the present paper). Here we mention that in \cite{AD+IK+MK+VV-20} 
a prediction made at the beginning is only used, that is, $p_k = p_0$ for all $k$. Moreover, the result in 
\cite{AD+IK+MK+VV-20} only applies to the Euler method for approximate solving the auxiliary problems involved in the MPC
algorithm.
}\end{rem}

\begin{rem} \label{R_approx} About the approximate solution of the auxiliary problems $\p_{p}(\tau, x_\tau)$.
{\em Finding an approximate solution of the auxiliary problems is a separate issue that we do not address in detail in this paper.
Numerical solutions usually involve time-discretization. Discretization methods with first and second order accuracy
are known for coercive problems (see Subsection \ref{SS_coerc} for the last term), \cite{Dont-96,Dont+Hager+Veliov:RK}, 
as well as for affine problems with purely bang-bang optimal controls, \cite{WAlt+UF+MS-18,P+S+Veliov-17}. 
The error in solving the resulting mathematical programming problems comes in addition.
The above mentioned results are proved under assumptions that imply strong sub-regularity 
(for appropriate sets $\Gamma$) of the optimality maps associated with the considered problems.  
}\end{rem}

\begin{example} \label{Example} Sharpness of the estimate. {\em
Consider the problem 
\bd
        \min \{ x^1(1) - x^2(1) \},
\ed
\bda 
             \dot x^1(t) &=& p(t) x^2(t), \quad x^1(0) = 0, \\
             \dot x^2(t)& =& u(t), \qquad\quad\; x^2(0) = 0, \qquad u(t)) \in [-1,1].
\eda
The reference parameter is $\hat p \equiv 1$. The measurements are assumed exact, as well as the solutions of the 
auxiliary problems at the MPC steps. Thus $\E = h \sum_{k=0}^{N-1} \| p_k - \hat p \|_\infty$.

The solution of each problem $\p_p(t_k, \hat x(t_k))$ is straightforward: here 
\bd
         \tilde \lambda_k^2( t ) = -1 + \int_t^1 p_k(s) \dd s, \qquad \tilde u_k ( t) = - {\rm sign }\, \tilde\lambda^2(t).
\ed
For $\hat p$ we have $\hat \lambda^2( t ) < 0$  for all $t \in (0,1]$, hence $\hat u(t) \equiv 1$. 
For $t = 1$ the control function $\hat u$ is not determined
by the Pontryagin necessary optimality condition, because $\hat \lambda^2( 0) = 0$. 
This does not matter from the control perspective, but suggests to define $\Gamma = \{ 0 \}$  (see Remark \ref{R_A3}).
As it will become obvious in the next section, Assumption (A3) is fulfilled for this problem with this $\Gamma$.

Let us fix an arbitrary $\delta > 0$ and consider $h= 1/N$ with $N > 2$ and such that $2h \leq \delta$.
Define $p_0 = 1+ 2h$ and take all other predictions exact: $p_k = 1$, $k = 1, \ldots, N-1$.
Then $\|p_0 - \hat p\|_\infty \leq 2h \leq\delta$. Moreover,  $\E = 2h^2$.

On the other hand, we have
\bd
           \tilde \lambda^2_0(t_{1}) = -1 + (1-h)  (1+2h) = h(1-2h) > 0.
\ed
Since $\tilde \lambda^2_0$ is linear and $\tilde \lambda^2_0(1) = -1$, 
we obtain that $\tilde \lambda^2_0(t) > 0$ on $[0,t_1]$.
Hence, $u^N(t) = \tilde u_0(t) = -1$ on $[0,t_1]$.  Then 
\bd
           \|u^N - \hat u \|_1 \geq 2h.
\ed
Consequently,
\bd
          \frac{ \|u^N - \hat u \|_1}{\sqrt{\E}} \geq \frac{2h}{\sqrt{2}h} \geq \sqrt{2}.
\ed
Since $\E$ can be arbitrarily small (for small $h$), the estimation in the theorem is sharp. 
}\end{example}

\section{Sufficient conditions for strong sub-regularity of the optimality map} \label{S_conditions}

In this section we present some classes of problems for which Assumption (A3) has a more particular form
with a specified set $\Gamma$, thus Theorem \ref{Tmain} is applicable. In addition, we further discuss the 
approximation issue mentioned in Remark \ref{R_approx}.

\subsection{The case of coercive problems} \label{SS_coerc} 

Following \cite{DH-93}, in this subsection we consider the reference problem $\p_{\hat p}(0,x_0)$ 
under the so-called {\em coercivity condition}. To formulate it we use the following notational convention:
we skip arguments of functions with ``hat'', shifting the ``hat'' over the notation of the function, e.g.
$\hat f_x(t) := f_x(\hat p(t), \hat x(t), \hat u(t))$, $\hat H_{xx}( t) := H_{xx}(\hat p(t), \hat x(t), \hat \lambda(t),\hat u(t))$, etc.
Here $f_x$ and $H_{xx}$ are the derivative of $f$ and the Jacobian of $H$, respectively.

\bino
{\bf Assumption (B1).} There is a constant $c_0 > 0$ such for any $v \in \U - \U$ the inequality 
\bd
      \langle g_{T}''(\hat x(T))x(T),x(T)\rangle +   \int_0^T \left[ \lll \hat H_{xx}(t) x(t), x(t)  \rrr  + 2 \lll \hat H_{ux}(t) x(t), v(t)  \rrr + \lll \hat H_{uu}(t) v(t), v(t)  \rrr \right] \dd t  
                \geq  c_0 \| v \|_2^2
\ed 
is fulfilled, where $x$ is the (unique) solution of the equation 
$\dot x(t) = \hat f_x(t) x(t) + \hat f_u(t) v(t)$ with $x(0) = 0$.
 
\bino
It was proved in \cite{DH-93} that Assumption (B1), together with (A1) and (A2), implies (A3)
with $\Gamma = \emptyset$, thus in this case the metric in $\U$ is $d^* = \| \cdot \|_\infty$, 
thus the metric in $Y$ can be taken to be $d^*(y,0) = \| x \|_{1,\infty} + \| \lambda \|_{1,\infty} +  \| u \|_{\infty}$.%
\footnote{%
\label{FnDH} The terminology of metric regularity was not used in \cite{DH-93} and the control system considered was 
stationary, but the result was easily extended to the non-stationary case in many subsequent contributions, see e.g.
\cite{Don+Kra+Vel-18}. }

Even more, Assumptions (A1), (A2), (B1) imply the stronger property of {\em Strong metric Regularity} (SR),
\cite[Sect.~3.7]{AD+TR-98} with respect to the norm $\| x \|_{1,\infty} + \| \lambda \|_{1,\infty} +  \| u \|_{\infty}$ 
in the space $Y$. An important fact is, that the property SR is stable with respect to functional perturbations with a 
sufficiently small Lipschitz constant see, e.g.,  \cite[Proposition 3G.2]{AD+TR-98}). Then it is easy to see that 
for sufficiently small inaccuracies $|e_k|$, $e_k^p$ and $e_k^{u}$ all the maps $\Phi_{(p,t_k,x_k^0)}$ 
that appear in the MPC algorithm are SR, 
hence Ss-R, with constants that can be chosen uniformly with respect to $k$. 
In connection with Remark \ref{R_approx}, we mention that thanks to the strong regularity of the optimality map 
(or the uniform strong sub-regularity)  one can claim $O(h)$ uniform estimation of $e_k^u$ if the Euler discretization
with step size $h$ is used in solving the problems $\p_p(t_k,x_k^0)$ (see \cite{AD+VV:CC-09}), and uniform 
convergence of the Newton method (see \cite{Aragon+...+VV-11}). 
However, this issue is not at the focus of the present paper and we do not give precise formulations and details.

\subsection{The case of affine problems with bang-bang optimal controls} \label{SSb-b} 

In this subsection, we consider problem $\mathcal P_{\hat p}(0,x_0)$ to be affine, i.e., 
the objective integrand, $g$, and the right-hand side, $f$, in (\ref{Ex}) are both affine with respect to $u$.  
The set $U$ is assumed to be a convex compact polyhedron. Using geometric terminology, 
we denote by $V$ the set of vertices  of $U$, and by $E$ the set of all unit vectors $e \in \R^m$ 
that are parallel to some edge of $U$.  
As usual, we define the so-called switching function $\hat\sigma:[0,T]\to\mathbb R^m$ by $\hat\sigma(t):=\hat H_u(t)$. 
Here and further we use the notational convention from  the previous subsection: arguments of functions with 
``hat'' are skipped and the ``hat'' is shifted over the notation of the functions.

Versions of the following assumption are standard in the literature on affine optimal control problems, 
see, e.g., \cite{WAlt+UF+MS-18,CVQ,Fel2003,Os+Ve-20}. 

\bino
{\bf Assumption (C1).}  There exist numbers $\eta_0 > 0$ and $\mu_0>0$ such that if $s \in[0,T]$ is a zero 
of $\langle \hat\sigma, e \rangle$ for some $e \in E$, then 
\begin{align*}
      |\langle \hat\sigma(t), e\rangle|\ge\mu_0|t-s|,
\end{align*}
for all $t\in [s-\eta_0,s+\eta_0] \cap [0,T]$.

\bino
Assumption (C1) implies that $\hat u$ is bang bang and that, in particular, the set 
\bd
         \Gamma:=\left\{ s \in[0,T] \sth \langle\hat \sigma(s),e \rangle=0 \, \text{ for some }\, e\in E \right\}
\ed
is finite. In what follows in this subsection,  the metric $d^*$ in $Y$ is defined through this set $\Gamma$,
see (\ref{EGam0}).
As it will be seen in the proof of Theorem \ref{Tmain}, the advantage of using this metric instead of 
the $L^1$-norm is that $d^*(u,\hat u)$ being small not only implies that $\| u - \hat u\|_1$
is small, but also that $|u(t) - \hat u(t)|$ is small except on a small set around the zeros of the switching functions
$\langle \hat\sigma(t), e\rangle$, $e \in E$. In this sense, $u$ is structurally similar to $\hat u$.

Given $\e \geq 0$, we  denote
\bd
         \Sigma(\e) := [0,T] \sm (\Gamma + [-\e, \e]).
\ed 

We recall the following lemma proved in the recent paper \cite{Sozopol}.

\begin{lem}{\cite[Lemma 2]{Sozopol}} \label{LH3}
Let Assumption (C1) be fulfilled. Then there exist  positive numbers $\kappa$ and $\e$ such that for every functions 
$\sigma \in L^\infty$ with $\|\sigma-\hat\sigma\|_\infty\le \e$ and for every $u \in \U$
satisfying $\sigma(t) + N_U(u(t)) \ni 0$ for a.e. $t \in [0,T]$ it holds that 
\bd
	u(t) = \hat u(t) \;\mbox{  for a.e. } \, t \in \Sigma\big(\kappa \| \sigma - \hat \sigma \|_\infty \big).
\ed
\end{lem} 

With this lemma at hand, we are now ready to establish the following sufficient condition for the fulfillment of Assumption (A3). 

\begin{theo}
Let problem $\mathcal P_{\hat p}(0,x_0)$ be affine and let Assumptions (A1), (A2), and (C1) be fulfilled. Then the following
statements are equivalent:

(i) The map $\Phi_{(\hat p,0,x_0)} : Y \To Z$ is strongly sub-regular at $(\hat y,0)$ in the single metric $d$ in $Y$; 

(ii)  The map $\Phi_{(\hat p,0,x_0)} : Y \To Z$ is strongly sub-regular at $(\hat y,0)$ in the metrics $d$ and $d^*$ in $Y$.
\end{theo}

\begin{proof}{}
The implication (ii) $\Mt$ (i) is obvious. Let us prove the converse implication.

Let $\tilde\alpha,\tilde\beta$ and $\tilde\kappa$ be the parameters of Ss-R of $\Phi_{(\hat p,0,x_0)}$ 
(in the single metric $d$ in $Y$). Then 
\begin{align}\label{smsrl1}
	 \| x -\hat x\|_{1,1} + \| \lambda-\hat\lambda \|_{1,1} + \| u-\hat u \|_1\le\tilde\kappa d_Z(z,0)
\end{align}
for all $y=(x,\lambda,u)\in B(\hat y;\tilde\alpha)$ and $z=(\xi, \nu, \eta,  \rho)\in B_Z(0,\tilde\beta)$ satisfying 
$z\in \Phi_{(\hat p,0,x_0)}(y)$. Let us fix arbitrarily such a pair $(y,z)$.
Define $\sigma:[0,T]\to \mathbb R^m$ by $\sigma(t):=\nabla_u H(\hat p,y(t))-\rho$.
Clearly $\sigma(t) + N_U(u(t)) \ni 0$ for a.e. $t \in [0,T]$. Moreover, due to the affine structure of the problem, 
$\nabla_u H(p,y)$ is independent of $u$, thus we can estimate 
\bda
        \|\sigma-\hat\sigma\|_{\infty} &\leq& \|\nabla_u H(\hat p,y) - \nabla_u H(\hat p,\hat y)\|_\infty + \| \rho \|_\infty 
      \leq L_1(  \| x -\hat x\|_\infty + \| \lambda-\hat\lambda \|_\infty ) + \| \rho \|_\infty \\
      &\leq&  L_1 \tilde \kappa d_Z(z,0) + d_Z(z,0) =: \tilde c d_Z(z,0),
\eda
where $L_1$ is the Lipschitz constant of $\nabla_u H(\hat p,\cdot)$.

Define 
\bd
         \hat\alpha=\tilde\alpha, \quad \hat\beta=\min\left\lbrace \tilde\beta, \e /\tilde c\right\rbrace, \quad
           \hat\kappa=\tilde\kappa+\tilde c\kappa,
\ed
where $\e$ and $\kappa$ are the numbers in Lemma \ref{LH3}. For the pair $(y,z)$ we additionally assume that 
$z\in B_Z(0,\hat\beta)$. Then by Lemma \ref{LH3}, 
\bd
	u(t) = \hat u(t) \;\mbox{  for a.e. } \, t \in \Sigma\big(\kappa\tilde cd_Z(z,0) \big),
\ed
which directly implies $d^*(u,\hat u)\le\kappa\tilde cd_Z(z,0)$. Together with (\ref{smsrl1}), we obtain that
\begin{align*}
         d^*(y,\hat y)\le\tilde c\kappa d_Z(z,0) + \tilde \kappa d_Z(z,0) = \hat \kappa  d_Z(z,0),
 \end{align*}
which proves (ii)  with Ss-R constants $\hat \al$, $\hat \beta$, $\hat \kappa$.
\end{proof}

A general sufficient condition for strong sub-regularity of the map  $\Phi_{(\hat p,0,x_0)} : Y \To Z$ 
in the single metric $d$ in $Y$ is given in \cite[Theorem 3.1]{Os+Ve-20}. It involves the following assumption.

\bino
{\bf Assumption (C2).} There is a constant $c_0 > 0$ such for any $v \in \U - \hat u$ the inequality 
\be \label{Ecoerc1}
      \int_0^T\lll \hat H_{u}(t), v(t) \rrr \dd t +  \langle g_{T}''(\hat x(T))x(T),x(T)\rangle +   \int_0^T \left[ \lll \hat H_{xx}(t) x(t), x(t)  \rrr  + 2 \lll \hat H_{ux}(t) x(t), v(t)  \rrr  \right] \dd t  
      \geq  c_0 \| v \|_1^2
\ee
is fulfilled, where $x$ is the (unique) solution of the equation $\dot x(t) = \hat f_x(t) x(t) + \hat f_u(t) v(t)$ with $x(0) = 0$.

\bino
In \cite[Theorem 3.1]{Os+Ve-20} it is proved that
Assumption (C2), together with (A1), (A2), and the affine structure of 
the problem, implies metric sub-regularity of the optimality map in the single metric $d$ in $Y$.
In contrast to the $L^2$ coercivity condition in the previous subsection, Assumption (C2) requires 
``coercivity'' with respect to the $L^1$-norm. 
It is well known that Assumption (B1) does not hold for affine problems, see \cite[Lemma 3]{Don+Kra+Vel-18}. 
Another difference between (B1) and (C2) is that the inequality in (C2) involves not only a quadratic, but also a
linear form of $v$. It is remarkable that alone this linear term can ensure fulfillment of Assumption (C1). 
Indeed, in \cite[Proposition 4.1]{Os+Ve-20} it is proved that Assumption (C1) implies the inequality
\bd
      \int_0^T\lll \hat H_{u}(t), v(t) \rrr \dd t \geq  c_1 \| v \|_1^2
\ed 
for a constant $c_1 > 0$ and all $v \in \U - \hat u$. In particular, if the quadratic form in (\ref{Ecoerc1}) is nonnegative 
for $v \in \U - \hat u$, then (A1), (A2), (C1) imply (C2), hence also Assumption (A3). 


\subsection{A numerical example}
In this section we illustrate the result obtained in Theorem \ref{Tmain} by considering a problem 
of axisymmetric spacecraft spin stabilization from \cite[p. 353]{Wie-98}. 
The transversal angular velocity components $\omega_1$ and $\omega_2$ of the spacecraft satisfy
\begin{equation*}
\begin{aligned}
      &\dot{\omega_1} = \lambda \omega_2 + \frac{M_d}{J_t}, \\
      &\dot{\omega_2} = -\lambda \sin\omega_1 + \frac{M_c}{J_t},
\end{aligned}
\end{equation*}
where $\lambda = \frac{J_t - J_3}{J_t}n$, $J_t$ is the spacecraft transversal moment of inertia, 
$J_3$ is the spacecraft moment of inertia about the spin axis, $n$ is the spin rate, 
$M_d$ is the disturbance torque, which can be caused by thruster misalignment, 
and $M_c$ is the control moment. Rescaling the time $(t \rightarrow \lambda t)$, denoting 
$x_1 = \omega_1$, $x_2 = \omega_2$, $p =\frac{M_d}{J_t}$, 
adding initial conditions, and considering $u := \frac{M_{c}}{J_t}$ as a control variable, 
we reformulate the model as
\begin{align}\label{Ex12}
 \left\{ \begin{array}{ll}
     \dot{x}_1 = x_2 + p, \quad & x_1(0) = 1,\\
     \dot{x}_2 = -\sin x_1 + u,  \quad  & x_2(0) = 1,\\
       -a \leq u \leq a,
\end{array}
\right.
\end{align}
where $p(\cdot)$ is a time-dependent parameter and $a$ is a positive constant.
The MPC algorithm is applied in \cite{AD+IK+MK+VV+PV-20} to the following optimal control problem 
with the dynamic (\ref{Ex12}):
\bd
              \text{min} \Big\{ |x(T)|^2 + \alpha \int_0^T (u(t))^2 \dd t \Big\},  
\ed
where $\al$ is a positive weighting parameter. 
This problem is coercive in the sense of Assumption (B1), which makes the analysis in  
\cite{AD+IK+MK+VV+PV-20} possible. Here we consider the alternative objective functional 
\be\label{EOF1}
       \text{min} \Big\{ |x(T)|^2 + \alpha \int_0^T |u(t)| \dd t \Big\},
\ee
which may be more realistic in case of direct transformation of fuel into force, as in jet engines. 
The optimal control problem (\ref{Ex12})-(\ref{EOF1}) is not coercive, nor does it fit in
the framework of affine problems. However, following \cite[Remark 3.3]{Sethi}, we transform 
it to an affine problem by substituting 
\begin{align*}
	u = u_1 - u_2, \quad  |u|=u_1+ u_2, \quad \text{where} \quad u_1, u_2 \in [0,1]. 
\end{align*} 
Thus, the affine optimal control problem we will consider is
\bd
\begin{aligned}
       \text{min} \Big\{ |x(T)|^2 + \alpha \int_0^T [u_1(t)+u_2(t)]\, \dd t \Big\},
\end{aligned}
\ed 
subject to
\bda
      \left\{ \begin{array}{llc}
        \dot{x}_1 = x_2 + p, \quad  &x_1(0) = 1,\\
        \dot{x}_2 = -\sin x_1 + u_1-u_2, \quad  &x_2(0) = 1,\\
              0 \leq u_1,u_2 \leq a.
\end{array}
\right.
\eda
We consider the last problem with the specifications $T=4\pi$, $\alpha = 0.25$, $a = 0.2$, 
and reference parameter $\hat{p}(t) \equiv 0$. In the MPC simulation,  
the measurement error $e_k$ is sampled randomly from a uniform distribution, with $|e_k|\le 0.1$.
The parameter $p(\cdot)$ is piecewise constant on the uniform mesh of $3200$ points in $[0,T]$; 
its values in every subinterval are chosen randomly in the interval $[-0.05, 0.05]$ with uniform distribution.  
For solving the auxiliary problems $\p_{p}(t_k, x_k^0)$ we use the Euler discretization scheme, 
which provides an error $e_k^{u}$ of order $O(h)$ see, e.g., \cite{WAlt+UF+MS-18, Os+Ve-20}); 
recall that $e_k^{u}:=\|\tilde z_k\|$ and that $\tilde z_k$ is the residual in (\ref{LRrez}).

We run the MPC algorithm with different mesh sizes $N$.
Using the notations in Theorem \ref{Tmain}, we consider the quantity
\begin{equation}
	\begin{aligned}
	RE = \frac{||u^N - \hat{u}||_1 + ||x^N - \hat{x}||_{1,1}}{\sqrt{\frac{1}{N}
             \sum_{k=0}^{N-1}(|e_k| + e_k^p + h)} + h},
	\end{aligned}
\end{equation}
which represents the relative error in the MPC-generated solution $(x^N, u^N)$. 
According to the error estimate in Theorem \ref{Tmain}, the quantity $RE$ should be bounded.	
The numerical experiment confirms this, as can be seen in Table \ref{table:1}.
Moreover, the result suggest that the value $RE$ stays away from zero when $N$ increases, 
which indicates that the estimate in Theorem \ref{Tmain} is sharp for this example.  

We also observe in Table \ref{table:1} that the objective values for the MPC-generated solutions
decrease when $N$ increases, which is to be expected because of the more frequent measurements.

In Figure \ref{fig:exp}, we compare the obtained MPC-generated controls and the corresponding trajectories 
with the open-loop solution $\hat u = \hat u_1 - \hat u_2$. The auxiliary controls $\hat u_1$ and $\hat u_2$ are of bang-bang
type, while the resulting optimal control $u$ in problem (\ref{Ex12})-(\ref{EOF1}) also takes value zero. 
The MPC-generated control $u^N$ differs from the optimal open-loop one in small intervals around the switching points 
of the latter, which is consistent with the choice of the metric in $d^*$ in case of affine problems.

\begin{table}[h!]
	\centering
	\begin{tabular}{||c  ||c c c c c c||} 
		\hline
		N & 160 & 320 & 480 & 640 & 800 & 960 \\ [0.5ex]
		\hline\hline
		Obj. val. & 1.3221 & 0.7907 & 0.7204 & 0.6490 & 0.6274 & 0.6183 \\
    $RE$      & 0.6039 & 0.3205 & 0.2177 & 0.1245 & 0.1445 & 0.1346 \\  [1ex] 
		\hline
	\end{tabular}
	\caption{Objective values and relative errors $RE$ of the MPC-generated solutions with different mesh sizes.}
	\label{table:1}
\end{table}

\begin{figure}[h!]
		\centering
		\begin{subfigure}[b]{0.48\textwidth}
			\centering
			\includegraphics[width=\textwidth]{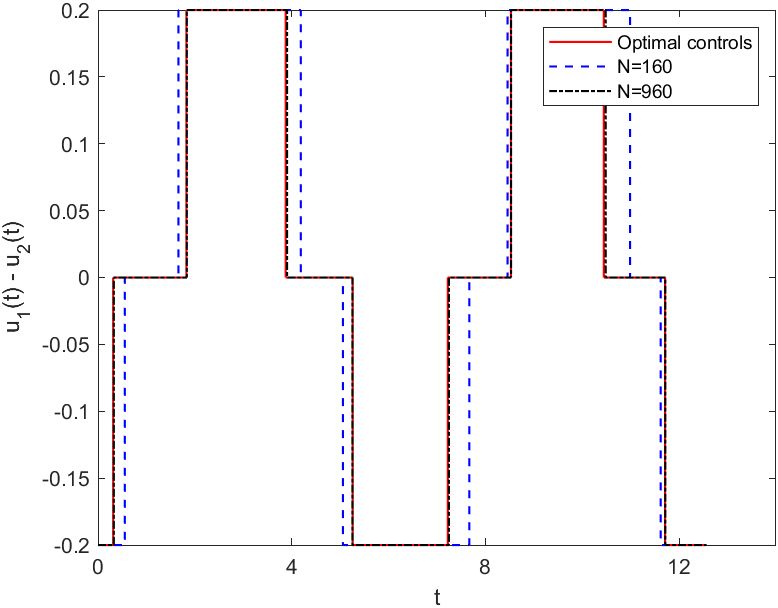}
			\caption{Control solutions}
		\end{subfigure}
		\hfill
		\begin{subfigure}[b]{0.48\textwidth}
			\centering
			\includegraphics[width=\textwidth]{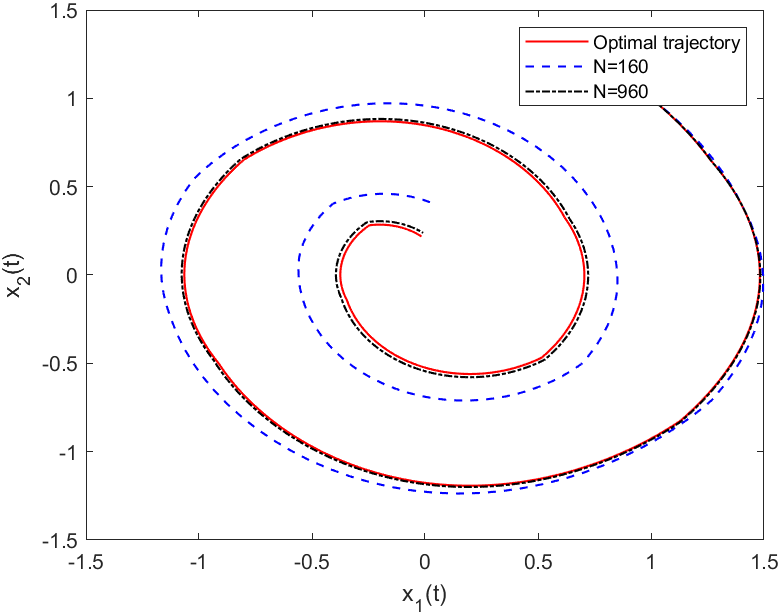}
			\caption{Trajectories}
		\end{subfigure}
	\caption{Red: optimal open-loop control and trajectory; Blue and Black: MPC generated solutions 
            with N=160 and N=960, correspondingly.}
		\label{fig:exp}
\end{figure}

\section{Proof of the main theorem} \label{Sproof}

In this section we use the notations introduced in Sections \ref{SPrelim} and \ref{S_MPC}.
In addition, $y \in Y_\tau$ we define
\bd
          d_\tau^*(y, \hat y) := d^*(\tilde y, \hat y), \quad \mbox{ where } 
               \tilde y(t) = \left\{ \begin{array}{cl}
                            \hat y(t) & \mbox{ for } t \in [0,\tau), \\ 
                             y(t)   & \mbox{ for } t \in [\tau, T].
                                \end{array} \right. 
\ed

\begin{prop} \label{Pmain_error}
Let assumptions (A1)--(A3) be fulfilled. Then there exist numbers $\delta_0 > 0$, $\al_0 > 0$, $\beta_0 >0$,  
and $c_0$ such that for every $\tau \in [0,T)$, $x_\tau \in \R^n$ and $p \in \varPi$ satisfying
\be \label{Edelta0}
            | x_\tau - \hat x(\tau)| \leq \delta_0, \qquad \|p - \hat p \|_\infty \leq \delta_0,
\ee
and for every $y = (x, \lambda, u) \in Y_\tau$ with $\| u - \hat u \|_1 \leq \al_0$ and 
$z_\tau \in \Phi_{(p,\tau,x_\tau)}(y) \cap \B_{Z_\tau}(0;\beta_0)$ it holds that 
\bd
                  d_\tau^*(y,\hat y) \leq c_0 \left( \| z_\tau \| + \| p - \hat p \|_\infty + | x_\tau - \hat x(\tau) | \right).
\ed
\end{prop}

\begin{proof}{}
We shall fix the numbers $\delta_0 > 0$, $\al_0 > 0$, $\beta_0 > 0$ and $c_0$ below, in a way that they 
depend only on $\hat \al$, $\hat\beta$, $\hat \kappa$, $L$ and $T$ and may be viewed as constants.
The numbers $c_1$, $c_2$, ..., that will appear later will also be appropriate constants in the  same sense.

Let us fix arbitrarily $\tau \in [0,T)$, $x_\tau \in \R^n$, and $p \in \varPi$ satisfying (\ref{Edelta0}), along with 
$y_\tau = (x(\cdot), \lambda(\cdot), u(\cdot)) \in \B(\hat y; \al_0)$ and 
$z_\tau = (\xi,\nu, \eta,\rho) \in \B_{Z_\tau}(0;\beta_0)$ 
such that $z_\tau \in \Phi_{(p,\tau,x_\tau)}(y_\tau)$. We define
\bd
               \tilde p(t) = \left\{ \begin{array}{cl}
                            \hat p(t) & \mbox{ for } t \in [0,\tau), \\ 
                             p(t)   & \mbox{ for } t \in [\tau, T],
                                \end{array} \right. \qquad
               \tilde u(t) = \left\{ \begin{array}{cl}
                            \hat u(t) & \mbox{ for } t \in [0,\tau), \\ 
                             u(t)   & \mbox{ for } t \in [\tau, T].
                                \end{array} \right.
\ed
Obviously $\| \tilde p - \hat p \|_\infty = \| p - \hat p \|_\infty$,  $\| \tilde u - \hat u \|_1 = \| u - \hat  u \|_1$, and 
$d_0^*(\tilde u, \hat u) =d_\tau^*(u,\hat u)$.  Similarly, we extend $\xi,\eta$ and $\rho$ as zero on $[0,\tau)$, denoting the 
resulting elements of $Z_0$ by $\tilde \xi,\tilde\eta$ and $\tilde \rho$.
Moreover, we define $\tilde x$ as the (unique) solution on $[0, T]$ of the equations
\bd
             \dot {\tilde x} = f(\tilde p, \tilde x, \tilde u) -\xi, \quad \tilde x(\tau) = x_ \tau + \nu,
\ed
and the function $\tilde \lambda$ as the  solution of 
\bd
            -\dot {\tilde \lambda} = \nabla_{\!\! x} H(\tilde p, \tilde x, \tilde \lambda,  \tilde u) -\eta, 
               \quad \tilde \lambda(T) = \nabla g_T(\tilde x(T)) + \pi.
\ed
Notice that $\tilde x(t) = x(t)$ and $\tilde \lambda(t) = \lambda(t)$ for $t \in [\tau,T]$.

Let us estimate $\| \tilde y - \hat y \|$. Using Assumption (A1) and the Gr\"{o}nwall inequality we obtain that 
\bd
         \| \tilde x - \hat x \|_{\infty} \leq c_1 \left(\| \tilde p - \hat p \|_\infty + \| \tilde \xi \|_1  + 
         | x_\tau - \hat x(\tau) | + |\nu | \right) \leq
                 c_2 \left( \delta_0 + \| \xi \|_1 + \delta_0 + | \nu |\right) \leq c_3 (\delta_0 + \|  z_\tau \|).
\ed
From here one can also estimate $\| \dot{\tilde x} - \dot{\hat x} \|_1$, which gives
\bd
          \| \tilde x - \hat x \|_{1,1} \leq c_4 (\delta_0 + \|  z_\tau \|).
\ed
Similarly we estimate
\bd
              \| \tilde \lambda - \hat \lambda \|_{1,1} \leq c_5 (\delta_0 + \|  z_\tau \|).
\ed
Moreover, 
\bd
             \| \tilde u - \hat u\|_1 =  \| u - \hat u\|_1 \leq \al_0.
\ed
The last three estimates imply that 
\bd
               \| \tilde y - \hat y \| \leq c_6 (\delta_0 + \alpha_0 + \beta_0) \leq \hat \al,
\ed
provided that the positive numbers $\delta_0$, $\al_0$ and $\beta_0$ are chosen sufficiently small.

Now we estimate the residual $r := (r_\xi, r_\nu, r_\eta, r_\rho) \in Z_0$ 
which $\tilde y := (\tilde x, \tilde \lambda, \tilde u)$ gives in $\Phi_{(\hat p,0,x_0)}$. We have
\bda
              \| r_\xi \|_1 &=& \big\| f(\hat p, \tilde x, \tilde u) - f(\tilde p, \tilde x, \tilde u) + \tilde \xi \big\|_1
                  \leq  T L \| p - \hat p \|_\infty  + \| \xi \|_1, \\
              | r_\nu | &= & | \tilde x(0) - x_0 | \leq c_1( \| p - \hat p \|_\infty +  |x_\tau - \hat x(\tau)| + | \nu | ), \\
              \| r_\eta \|_1 &=&  \big\| \nabla_{\!\! x} H (\hat p, \tilde y) -  
                  \nabla_{\!\! x} H(\tilde p, \tilde y) + \tilde \eta \big\|_1 \leq  T L \| p - \hat p \|_\infty  + \| \eta \|_1, \\
                  |r_\pi| = |\pi|,\\ 
                \| r_\rho \|_\infty &=&  \| \nabla_{\!\! u} H (\hat p, \tilde y) -  
                  \nabla_{\!\! u} H(\tilde p, \tilde y) + \tilde \rho \|_\infty \leq  L \| p - \hat p \|_\infty  + \| \rho \|_\infty.
\eda
Summarizing, we obtain that 
\bd
               \| r \| \leq c_7 \left( \| z_\tau \| + \| p - \hat p \|_\infty + | x_\tau - \hat x(\tau) | \right).
\ed     
We can choose $\delta_0$ and $\beta_0$ smaller if needed so that $\|r\|\le\hat\beta$. Due to Assumption (A3) we have that
\bd
             d^*_0(\tilde y, \hat y) \leq \hat \kappa  c_7 \left( \| z_\tau \| + \| p - \hat p \|_\infty + | x_\tau - \hat x(\tau) | \right).
\ed
Since $d^*_\tau(\tilde y, \hat y) \leq d^*_0(\tilde y, \hat y)$ and $\tilde y = y$ on $[\tau,T]$, 
we obtain the desired result with $c_0 = \hat \kappa c_7$.
\end{proof}

An alternative way to prove the last proposition is first to show that Assumption (A3) holds for all maps 
$\Phi_{(\hat p,\tau, \hat x(\tau))}$ with $\tau \in [0,T)$, and then to apply Proposition \ref{P_LG}.

\bino
Now we continue with the proof of Theorem \ref{Tmain}.

Let $\delta_0$, $\al_0$, $\beta_0$ and $c_0$ be the constants from Proposition \ref{Pmain_error}.
Define the following constants: $M$ is number of elements of $\Gamma$ (equals zero if $\Gamma = \emptyset$) and
\be \label{Econst}
    D := \mbox{ diam}(U), \qquad 
           \bar C_1 := c_0 L T e^{LT}, \qquad \bar C_2 := 2 D L e^{TL} \sqrt{c_0TM}, 
             \qquad   \bar C_3 := 6 M D L e^{LT}.
\ee
Let the numbers $N_0$ and $\delta > 0$ be defined in such a way that
\bd
            \hat \delta + \delta \leq \delta_0, \qquad  \delta \leq \al_0, \qquad \delta \leq \beta_0,
\ed
where $\hat \delta :=\bar C_1 \delta + \bar C_2 \sqrt{\delta} + \bar C_3 h$,
which is obviously possible.
Moreover, denote 
\bd
                {\cal E}_i := |e_i| + e^p_i + e^{u}_i,   \;\; i = 0, \ldots, N-1.
\ed
Since for any $i \in \{ 0, \ldots, N-1 \}$ the triplet $y = \tilde y_i$ satisfies (\ref{LRrez}), we shall apply 
Proposition~\ref{Pmain_error} 
with $y = \tilde y_i$, $\tau = t_i$ $x_\tau = x_i^0 = x^N(t_i) + e_i$, $p = p_i$, $z_\tau = \tilde z_i$. We have
\bd
                  \| p_i - \hat p \|_\infty \leq e^p_i \leq \delta \leq \delta_0,
\ed
\bd
            \| \tilde u_i - \hat u \|_1  \leq \delta \leq \alpha_0,
\ed
\bd
             \| \tilde z_i \| = e_i^{u} \leq \delta \leq \beta_0.  
\ed
Proposition \ref{Pmain_error} gives 
\be \label{EHu}
            d^*_{t_i}(\tilde u_i, \hat u)  \leq c_0 (|e_i| + e^p_i + e^{u}_i) =  c_0 \E_i,
\ee
provided that 
\be \label{EH61}
               | x^N(t_i) + e_i - \hat x(t_i) | \leq \delta_0. 
\ee

Let us fix an arbitrary $k \in \{1, \ldots, N-1 \}$ and denote
\bd
     d_i := \| \tu_i - \hu \|_{L^\infty(t_i,t_{i+1})}, \qquad 
        \Delta(t) := |x^N(t) - \hat x(t)|, \qquad \Delta_i = \Delta(t_i), \;\; i = 0, \ldots, k.
\ed 
Assume inductively that 
\be \label{EDxu}
    \Delta_k \leq \hat \delta \quad \mbox{ and } \;  d^*_{t_i}(\tilde u_i, \hat u) \leq c_0 \E_i, \;\; i=0, \ldots, k.
\ee
For $k=0$ we have $\Delta_0 = |x^N(0) - \hat x(0)|  = 0$. Thus (\ref{EH61}) is fulfilled because 
$|e_0| \leq \delta \leq \delta_0$. 
The second inequality in (\ref{EDxu}) is fulfilled due to (\ref{EHu}), thus
the inductive assumption is fulfilled for $k = 0$.

Due to Assumption (A1) and the construction of
$u^N$ in the MPC method, we have for $t \in [t_k,t_{k+1}]$
\bd
      \Delta(t)  \leq \Delta_k + \int_{t_k}^t | f(\hat p(s), x^N(s), \tilde u_k(s)) -   f(\hat p(s), \hat x(s), \hu(s)) | \dd s
      \leq \Delta_k + \int_{t_k}^t (L \Delta(s) + L | \tu_k(s) - \hu(s)|) \dd s.
\ed
Using the Gr\"{o}nwall inequality we obtain that 
\bd
            \Delta(t) \leq e^{Lh} (\Delta_k + hLd_k).
\ed
Applied to $\Delta_{k+1} = \Delta(t_{k+1})$, this recursive inequality implies in a standard way that 
\be \label{EDel_k}
             \Delta_{k+1} \leq hL\Big(e^{(k+1)hL} d_0 + e^{khL} d_1 + \ldots + e^{hL} d_{k} \Big) 
                    \,\leq \, e^{TL} L h \sum_{i=0}^{k} d_i.
\ee
 
The key part of the proof is to estimate $\sum_{i=0}^{k} d_i$. Let us denote
\bda
       \bar K &:=& \{i \in \{0, \ldots, k \} \sth | \tu_i (t) - \hu(t) | \leq d^*_{t_i}(\tu_i, \hu) 
           \mbox{ for a.e. } t \in [t_i, t_{i+1}] \}, \\
          K & := & \{ 0, \ldots, k\} \setminus \bar K.
\eda
Then
\be \label{EH12}
          d_i \leq \left\{ \begin{array}{cl}
                               D & \mbox{ for } i \in K, \\
                               d_{t_i}^*(\tilde u_i,  \hat u) & \mbox{ for } i \in \bar K.
                 \end{array} \right.
\ee
Denoting by $s$ the number of elements of $K$, we have due to the inductive assumption (\ref{EDxu}), that
\be \label{EH75}
    \sum_{i=0}^{k} d_i =  \sum_{i \in K} d_i + \sum_{i \in \bar K} d_i \leq 
                   s \, D + \sum_{i \in \bar K} d^*_{t_i}(\tu_i, \hu) \leq s \, D + \sum_{i \in \bar K} c_0 \E_k
             \leq s \, D +  \frac{c_0 T}{h} \E .
\ee
Let us assume that $\Gamma \not= \emptyset$, that is, $M > 0$.
The definition of the metric $d^*$ in (\ref{EGam0}) implies that for each $i \in K$ there exists 
$t \in (t_i, t_{i+1})$ such that 
\be \label{EH787} 
         \di(t, \Gamma) \leq d^*_{t_i}(\tu_i, \hu).
\ee
Let $m(i)$ be the minimal natural number (also including $0$) such that 
\bd
         \left( (t_i, t_{i+1})  + h [ -m(i), m(i) ] \right) \cap \Gamma \not= \emptyset.
\ed
Then in the case $m(i) > 0$ we have that 
\bd
         \left( (t_i, t_{i+1})  + h [ -m(i)+1, m(i) - 1 ] \right) \cap \Gamma = \emptyset,
\ed
hence
\bd
         \left( t + h [ -m(i)+1,m(i) - 1 ] \right) \cap \Gamma = \emptyset.
\ed
Due to (\ref{EH787}) we obtain that 
\be \label{Ed_m}
           d^*_{t_i}(\tu_i, \hu) \geq  h (m(i) - 1 ), \quad i \in K.
\ee
Denote by $l_j$, $j = 0, 1, \ldots, N$ the number of those indexes $i \in K$ for which $m(i) = j$. The following relations 
are apparently satisfied:
\bea \nonumber
              && 0 \leq l_0 \leq M, \\
              && 0 \leq l_j \leq 2M, \quad j = 1, \ldots, N,  \label{Elj} \\
               && \sum_{j=0}^N l_j = s.  \nonumber
\eea 

Then, having in mind (\ref{Ed_m}),
\be \label{EH44}
        \sum_{i \in K} d^*_{t_i}(\tilde u_i, \hat u) \,\geq\,  h \sum_{i \in K} \max\{ 0, m(i) - 1 ) \}  \geq 
                h\Big( 0.l_0 + 0. l_1 + \sum_{j=2}^N (j-1) l_j \Big).
\ee

The minimum of the sum in the right-hand side with respect to $\{l_j\}$ subject to the relations around (\ref{Elj})
is attained at
\bd
           l_0 = M, \; l_j = 2M, \; j = 1, \ldots, r, \;\;\; l_{r+1} = s - (M + 2Mr),   
\ed
where $r = \left[ \frac{s-M}{2M}\right]$ and $[a]$ means the integer part of $a$. Substituting $l_j$ in (\ref{EH44})
we obtain that
\bd
     \sum_{i \in K} d^*_{t_i}(\tilde u_i, \hat u) \,\geq\,  2 h M \sum_{j=2}^{r} (j-1) = hM r (r-1) \geq 
                hM \left( \left[ \frac{s}{2M}\right] -2\right)^2.
\ed
From here and the second inequality in (\ref{EDxu}) we obtain that
\bd
     \Big([\frac{s}{2M}] - 2\Big)^2 \leq \frac{c_0}{Mh} \sum_{i=0}^k \E_{i} = \frac{c_0 T}{Mh^2} \E,  
\ed
hence
\bd
       s \,\leq\, 6M +  \frac{2}{h} \sqrt{c_0TM \E}. 
\ed
From (\ref{EH75})  we obtain that 
\be \label{Esdi}
      h \sum_{i=0}^k d_i \,\leq\,  D \Big( 6Mh  + 2 \sqrt{c_0 TM} \sqrt{\E}\Big) +  c_0 T \E,
\ee
which combined with (\ref{EDel_k}) and (\ref{Econst}) gives
\be \label{EH55}
       \Delta_{k+1} \,\leq\, \bar C_1 \E + \bar C_2 \sqrt{\E} + \bar C_3 h.
\ee
This inequality was obtained in the case $M > 0$. However, in the case $M = 0$ the first term in the final inequality  
in (\ref{EH75}) is missing and the analysis simplifies, resulting in the same estimation for $\Delta_{k+1}$
but with $M = 0$, which implies $\bar C_2 = \bar C_3 = 0$.

Now, in order to verify  the inductive assumption (\ref{EDxu}) we observe that $\E \leq \delta$, hence from (\ref{EH55})
\bd
             \Delta_{k+1} \leq   \bar C_1 \delta + \bar C_2 \sqrt{\delta} + \bar C_3 h = \hat \delta,
\ed
thus the first inequality in  (\ref{EDxu}) is satisfied for $k+1$. 
The second inequality in (\ref{EDxu}) for $k+1$ follows from (\ref{EHu}), which is fulfilled for $i=k+1$
because (\ref{EH61}) holds: 
\bd
     |x^N(t_{k+1}) + e_{k+1} - \hat x(t_{k+1}) | \leq \Delta_{k+1} + |e_{k+1}| \leq \hat \delta + \delta \leq \delta_0.
\ed

This completes the induction step. As a result we have obtained that for any $t \in [0,T]$ if $t \in [t_k,t_{k+1}]$
then
\be \label{EDel_t}
              \Delta(t)  \leq \Delta_{k+1} \leq  \bar C_1 \E + \bar C_2 \sqrt{\E} + \bar C_3 h. 
\ee

Now we estimate
\bea
        \nonumber   \| u^N - \hat u\|_1 &=& \sum_{i=0}^{N-1} \int_{t_i}^{t_{i+1}}|u^N(t) - \hat u(t)| \dd t
     \,\leq \, \sum_{i=0}^{N-1} h \|u^N - \hat u\|_{L^\infty(t_i,t_{i+1})} 
            = \sum_{i=0}^{N-1} h \|\tilde u_i - \hat u\|_{L^\infty(t_i,t_{i+1})} 
    =  h \sum_{i=0}^{N-1} d_i \\ &\leq& D \Big( 6Mh  + 2 \sqrt{c_0 TM} \sqrt{\E}\Big) +  c_0 T \E, \label{EH59}
\eea
where in the last inequality we use (\ref{Esdi}) for $k = N-1$. 
Finally, we have
\bd
      \| \dot x^N - \dot{\hat x}\|_1 \leq \int_0^T | f(\hat p(t), x^N(t), u^N(t)) -  f(\hat p(t), \hat x(t), \hat u(t))| \dd t
       \leq LT \| \Delta \|_C + L \| u^N - \hat u \|_1.
\ed
Combining this inequality with (\ref{EDel_t}) and (\ref{EH59}), and considering separately the case $M = 0$,
we obtain existence of constants $C_1$, $C_2$ and $C_3$ for which the claim of the theorem holds.

\end{document}